\documentclass{amsart}
\usepackage{multicol,graphicx,color}
\usepackage{pslatex}
\usepackage{amsthm}
\usepackage{amsmath}
\usepackage{amssymb}
\usepackage{amsfonts}

\newcommand{\myfig}[3][0]{
\begin{center}
  \vspace{0.2cm}
  \includegraphics[width=#3\hsize,angle=#1]{#2}
  \nobreak\medskip
\end{center}
}

\newcommand{\mycaption}[1]{
  \vspace{0.2cm}
  \begin{quote}
    {{\sc Figure} \arabic{figure}: #1}
  \end{quote}
  \vspace{0.2cm}
  \stepcounter{figure}
}

\theoremstyle{plain}
\newtheorem{theorem}{Theorem}[section]
\newtheorem{corollary}[theorem]{Corollary}
\newtheorem{lemma}[theorem]{Lemma}

\begin{document}

\title[Linear Stability of SBC Orbits]{Linear Stability for Some Symmetric Periodic Simultaneous Binary Collision Orbits in the Four-Body Problem}
\author[Bakker]{Lennard F. Bakker}
\author[Ouyang]{Tiancheng Ouyang} 
\author[Simmons]{Skyler Simmons}
\author[Yan]{Duokui Yan}
\address{Department of Mathematics \\  Brigham Young University\\ Provo, UT 84602}
\email[Lennard F. Bakker]{bakker@math.byu.edu}
\email[Tiancheng Ouyang]{ouyang@math.byu.edu}
\email[Duokui Yan]{duokui@math.byu.edu}
\email[Skyler Simmons]{xinkaisen@yahoo.com}

\author[Roberts]{Gareth E. Roberts}
\address{Department of Mathematics and Computer Science \\ College of the Holy Cross \\ 1 College Street \\ Worcester, MA 01610}
\email[Gareth E. Roberts]{groberts@radius.holycross.edu}

\date{}

\thanks{The research of Gareth E. Roberts supported in part by NSF grant DMS-0708741}


\begin{abstract} We apply the analytic-numerical method of Roberts to determine the linear stability of time-reversible periodic simultaneous binary collision orbits in the symmetric collinear four body problem with masses 1, m, m , 1, and also in a symmetric planar four-body problem with equal masses. For the collinear problem, this verifies the earlier numerical results of Sweatman for linear stability.
\end{abstract}

\maketitle

\section{Introduction} Recently, Roberts \cite{RO} desribed an analytic-numerical method for determining the linear stability of a symmetric periodic orbit of a Hamiltonian system. He applied this method to the time-reversible collision-free figure-eight orbit in the equal mass three body problem numerically discovered by Moore \cite{MR} and whose existence was proven by Chenciner and Montgomery \cite{CM}. (Other such choreographic solutions were found numerically by Sim\'o \cite{Si}). Roberts' method shows that the figure eight orbit is linearly stable. The method uses the symmetries to factor a matrix similar to the monodromy matrix for the periodic orbit into an integer power of the product of two involutions. One of the two involutions depends on the linearized dynamics along only a part of the periodic orbit. For the figure eight this part is one-tweltfth of the full orbit since it has a symmetry group isomorphic to the group $D_3\times {\mathbb Z}_2$ of order $12$. (Here the dihedral group $D_k$ is the group of symmetries of the regular $k$-gon.) The eigenvalues of the product of the two involutions are then reduced to the numerical computation of a few real numbers.

Schubart \cite{Sc} numerically discovered a singular periodic orbit in the collinear equal mass three-body problem. The orbit alternates between binary collisions. H\'enon \cite{He} extended Schubart's numerical investigations to the case of unequal masses. Only recently did Venturelli \cite{Ve} and Moeckel \cite{Mo} prove the existence of the Schubart orbit when the outer masses are equal and the inner mass is arbitrary. The linear stability of the Schubart orbit was determined numerically by Hietarinta and Mikkola \cite{HM} revealing that linear stability occurs for some but not all of the choices of the three masses. Sweatman (\cite{SW} and \cite{SW2}) numerically found and determined the linear stability of a Schubart-like orbit in the symmetric collinear four body problem with masses $1$, $m$, $m$, and $1$. This Schubart-like periodic orbit alternates between simultaneous binary collisions (SBC) and inner binary collisions. Ouyang and Yan \cite{OY2} proved the existence of this orbit. In the regularized setting, this periodic orbit has a symmetry group isomorphic to $D_2$, of which both of the generators are time-reversing symmetries. Ouyang, Yan, and Simmons \cite{OY3} numerically found and then proved the existence of a singular periodic orbit in a symmetric planar four-body problem with equal masses in which the four bodies alternate between different simultaneous binary collisions. In the regularized setting, this periodic orbit has a symmetry group isomorphic to $D_4$, of which one of the generators is a time-reversing symmetry. The regularization of these singular periodic orbits is achieved by a generalized Levi-Civita type transformation and an appropriate scaling of time, as adapted from Aarseth and Zare \cite{AZ}.

In this paper we apply the method of Roberts to prove the linear stability of the Schubart-like orbit in the symmetric collinear four body $1$, $m$, $m$, $1$ problem for certain values of $m$, and of the singular periodic orbit in the symmetric planar equal mass problem. In both settings, the linear stability is determined for the regularized equations only and is reduced to the rigorous numerical computation of a single real number. Our linear stability analysis determines values of $m$ in the interval $[0,50]$ in the collinear problem for which the singular periodic orbit is linear stable, and also shows that the $2D$ singular periodic orbit is linear stable. These examples support and extend the conjecture made by Roberts \cite{RO} that the only linearly stable periodic orbits in the equal mass $n$-body problem are those that exhibit a time-reversing symmetry.

Our linear stability analysis confirms Sweatman's linear stability analysis \cite{SW2} for the singular periodic orbit in the collinear four-body problem. Sweatman used a numerical perturbation technique to assess the stability of the singular periodic orbit when the masses are arranged from left to right as $m_1$, $m_2$, $m_2$, and $m_1$ with the condition that $m_1+m_2=2$. Our mass parameter $m$ is related to his mass parameter $m_1$ by $m=(2-m_1)/m_1$. In terms of our mass parameter $m$, Sweatman's numerical results indicate that linear stability occurs when the value of $m$ is smaller than approximately $2.83$ and when it is larger than approximately $35.4$, and is linearly unstable otherwise.

\section{Linear Stability of Periodic Orbits}\label{-1} For a smooth function $\Gamma$ defined on an open subset of ${\mathbb  R}^{2n}$, suppose that $\gamma(s)$ is a $T$-periodic solution of a Hamiltonian system $z^\prime = J D\Gamma(z)$ where ${}^\prime = d/ds$,
\[ J = \begin{bmatrix} 0 & I \\ - I & 0\end{bmatrix},\]
and $I$ is the appropriately sized identity matrix. The fundamental matrix solution $X(s)$ of the linearized  equations along $\gamma(s)$,
\begin{equation}\label{1}
\xi^\prime= JD^2 \Gamma(\gamma(s)) \xi, \ \ \ \ \ \ \xi(0)=I
\end{equation}
is symplectic and satisfies $X(s+T)= X(s) X(T)$ for all $s$. The matrix $X(T)$ is commonly called the monodromy matrix for $\gamma$, and it measures the non-periodicity of solutions to the linearized equations. The eigenvalues of $X(T)$ are the characteristic multipliers of $\gamma$, and determine the linear stability of the periodic solution $\gamma$. Linear stability therefore requires that all of the multipliers lie on the unit circle.

The characteristic multipliers may be obtained by solving (\ref{1}) with different initial conditions. For an invertible matrix $Y_0$, let $Y(s)$ be the fundamental matrix solution to
\begin{equation}\label{2}
\xi^\prime = JD^2 \Gamma(\gamma(s)) \xi, \ \ \ \ \ \ \xi(0)=Y_0.
\end{equation}
By definition of $X(s)$, we know that $Y(s)=X(s) Y_0$, and so $X(T)=Y(T) Y^{-1}_0$. It follows that the matrix $Y_0^{-1} Y(T)$ is similar to the monodromy matrix i.e.,
\[ X(T)=Y(T) Y_0^{-1}= Y_0 (Y_0^{-1} Y(T)) Y_0^{-1}.\]
Thus the eigenvalues of $Y_0^{-1} Y(T) $ are identical to the characteristic multipliers.

\subsection{Stability reduction using symmetry}\label{0}
The monodromy matrix for a periodic solution with special types of symmetry can be factored using some linear algebra and standard techniques in differential equations. We begin by reviewing the relevant factorization and reduction theory that are applicable to a wide range of symmetric periodic orbits commonly found in Hamiltonian systems. Proofs of the following statements can be found in \cite{RO}.

\begin{lemma}\label{a}
Suppose that $\gamma(s)$ is a symmetric $T-$periodic solution of a Hamiltonian system with Hamiltonian $\Gamma$ and symmetry matrix $S$ such that:
 \begin{enumerate}
 \item for some positive integer $N$, $\gamma(s+T/N)= S  \gamma(s) $ for all $s;$
 \item $\Gamma(Sz)=\Gamma(z);$
 \item $SJ=JS;$
 \item $S$ is orthogonal.
 \end{enumerate}
Then the fundamental matrix solution $X(s)$ to the linearization problem in {\rm (\ref{1})} satisfies
 $$ X\left(s+T/N\right)= S X(s) S^{T} X(T/N).$$
 \end{lemma}

Here of course, the notation $S^T$ means the transpose of $S$. We mention this because we are using the letter $T$ in two distinct ways.

\begin{corollary}\label{aa}
Given the hypothesis of Lemma \ref{a}, the fundamental matrix solution $X(s)$ satisfies
\begin{equation*}
X(kT/N)= S^{k} \big(S^{T} X(T/N)\big)^{k}
\end{equation*}
for any  $k \in \mathbb{N}$.
\end{corollary}

A remark here is that if $Y(s)$ is the fundamental matrix solution to Equation (\ref{2}), then for any $k\in{\mathbb N}$, the matrix $Y(kT/N)$ factors as
\[ Y(kT/N)= S^{k} Y_0 (Y^{-1}_0 S^{T} Y(T/N))^{k}.\]

\begin{lemma}\label{b}
Suppose that $\gamma(s)$ is a $T-$periodic solution of a Hamiltonian system with Hamiltonian $\Gamma$ and time-reversing symmetry $S$ such that:
\begin{enumerate}
 \item for some positive integer $N$, $\gamma(-s+T/N)= S  \gamma(s) $ for all $s;$
 \item $\Gamma(Sz)=\Gamma(z);$
 \item $SJ=-JS;$
 \item $S$ is orthogonal.
 \end{enumerate}
Then the fundamental matrix solution $X(s)$ to the linearization problem in {\rm (\ref{1})} satisfies
\begin{equation*}
X(-s+ T/N)= S X(s) S^{T} X(T/N).
\end{equation*}
\end{lemma}

\begin{corollary}\label{c}
Given the hypothesis of Lemma \ref{b},
\begin{equation*}
X(T/N)= S B^{-1} S^{T} B \quad \text{where} \ B= X(T/2N).
\end{equation*}
\end{corollary}

Several more remarks about these factorizations are needed here.
\begin{enumerate}
\item In the case of time-reversing symmetry matrix, $S$ is typically block diagonal with two blocks of opposite sign, one for the position variable and one for the momenta, that is,
\[ \left[ \begin{array}{ccc}
F & 0  \\
0 & -F
\end{array}
\right] \]
where $F$ is orthogonal. A matrix of this form is orthogonal and anti-commutes with $J$.
\item A matrix satisfying properties 3 and 4 of Lemma \ref{b} is symplectic with a multiplier of $-1$ since $S^{T} J S= -S^{T} S J=-J$.
\item If $Y(s)$ is the fundamental matrix solution to (\ref{2}), then a similar argument shows that $Y(-s+T/N)= SY(s) Y_0^{-1} S^{T} Y(T/N)$ and consequently
    \begin{equation*}
    Y(T/N)= SY_0 B^{-1} S^{T} B, \quad \text{where} \ B=Y(T/2N).
    \end{equation*}
\end{enumerate}

Applying this factorization theory results in expressing the matrix $Y_0^{-1}Y(T)$, which is similar to $X(T)$, as $W^k$ for some positive integer $k$, where the symplectic matrix $W$ is the product of two involutions. If an eigenvalue of $W$ lies on the unit circle, then so does its $k^{\rm th}$ power. The symplectic matrix $W$ is called {\it stable} if all of its eigenvalues lie on the unit circle.

\begin{lemma}\label{06}
For a symplectic matrix $W$, suppose there is a matrix $K$ such that
\begin{equation}\label{5}
\frac{1}{2} \left( W + W^{-1}\right) =\left[
\begin{array}{ccc}
K^{T} & 0  \\
0 & K
\end{array}
\right]. 
\end{equation}
Then $W$ is stable if and only if all of the eigenvalues of $K$ are real and have absolute value smaller than or equal to $1$.
\end{lemma}

We will show for each of the periodic orbits under consideration, there is a choice of $Y_0$ such that $W$ satisfies Lemma \ref{06}. This reduces the linear stability to the computation of the eigenvalues of $K$. As one of the eigenvalues of $K$ is known to be real and have absolute value $1$, the linear stability is determined by the numerical computation of one real number and showing that, within error, it lies between $-1$ and $1$.

\section{Linear Stability for the Collinear Four-Body Symmetric Periodic Orbit}

The existence of the Schubart-like periodic orbit in the collinear four-body problem has been shown in \cite{OY2}. We review it here. For $x_1\geq x_2\geq 0$, we assume that four masses are located at $x_1$, $x_2$, $-x_2$ and $-x_1$ with masses $1$, $m$, $m$, and $1$ respectively with $m>0$. We also assume that the system remains symmetrically distributed about the center of mass located at the origin.  The respective velocities of the four bodies are $\dot{x}_1$, $\dot{x}_2$,$-\dot{x}_2$,$-\dot{x}_1$ where $\ \dot{} = d/dt$. The Netowanian equations are
\begin{equation*}
\ddot{x}_1= -\frac{1}{4 x_1^2} -\frac{m}{(x_1+x_2)^2} -
\frac{m}{(x_1-x_2)^2}
\end{equation*}

\begin{equation*}
\ddot{x}_2= -\frac{m}{4 x_2^2} -\frac{1}{(x_1+x_2)^2} +
\frac{1}{(x_1-x_2)^2}
\end{equation*}

We recount Sweatman's approach in \cite{SW} and \cite{SW2} to regularize this system. The Hamiltonian for this system is
$$H= \frac{1}{4} w_1^2+ \frac{1}{4m} w_2^2 - \frac{1}{2x_1} - \frac{m^2}{2x_2} - \frac{2m}{x_1+x_2}-\frac{2m}{x_1-x_2},  $$
where  $w_1= 2 \dot{x}_1$ and $w_2= 2m \dot{x}_2$ are the conjugate momenta to $x_1$ and $x_2$. Introduce new canonical coordinates $q_1,q_2,p_1,p_2$ by
$$q_1=x_1-x_2,  \ \ \ q_2=2 x_2, \ \ \  p_1=w_1,  \ \ \  p_2= \frac{1}{2}(w_1+w_2).$$
The Hamiltonian in the new canonical coordinates is
$$H= \frac{1}{4}\left(1+\frac{1}{m}\right)p_1^2   -\frac{p_1p_2}{m} + \frac{p_2^2}{m} -\frac{2m}{q_1}- \frac{m^2}{q_2} - \frac{2m}{q_1+q_2}- \frac{1}{2q_1+q_2}. $$
To regularize the equations of motion, Sweatman introduced a Levi-Civita type of canonical transformation
$$ Q _i^2= q_i ,\ \ P_i= 2 Q_i p_i \ \ (i=1,2),$$
for the the canonical coordinates $Q_1,Q_2,P_1,P_2$, and then replaced time $t$ by the new independent variable $s$ given by
\[ \frac{d t}{d s}= Q_1^2 Q_2^2.\]
In the extended phase space, this produces the regularized Hamiltonian
\begin{align*} \Gamma = \frac{dt}{ds}(H-E) & = \frac{1}{16}\left(1+\frac{1}{m}\right) Q_2^2 P_1^2 + \frac{-Q_1Q_2P_1P_2 + Q_1^2P_2^2}{4m} \\ 
& \ \ \ \ -m^2Q_1^2-2m Q_2^2 -\frac{2m Q_1^2 Q_2^2}{Q_1^2+ Q_2^2}-\frac{ Q_1^2 Q_2^2}{2Q_1^2+ Q_2^2} - E Q_1^2 Q_2^2.
\end{align*}
We fix the energy $E=-1$. The Hamiltonian system in the new coordinate system is
\begin{equation}\label{e1}
Q'_1= \frac{Q_2}{4} \left[\frac{1}{2}\left( 1+ \frac{1}{m}\right)Q_2 P_1 -\frac{1}{m}Q_1 P_2  \right],
\end{equation}

\begin{equation}\label{e2}
Q'_2= \frac{Q_1}{2m} \left[Q_1 P_2- \frac{1}{2}Q_2 P_1 \right],
\end{equation}

\begin{equation}\label{e3}
P'_1= \frac{P_2}{4m}( Q_2 P_1- 2Q_1 P_2)  + 2m^2Q_1 + \frac{4m Q_1 Q_2^4}{( Q_1^2+ Q_2^2)^2} + \frac{2 Q_1 Q_2^4}{(2 Q_1^2+ Q_2^2)^2} - 2Q_1 Q_2 ^2,
\end{equation}

\begin{equation}\label{e4}
 P'_2  = \frac{P_1}{4} \left[ \frac{Q_1 P_2}{m}-\frac{Q_2P_1}{2}\left(1+\frac{1}{m}\right)\right]
   + 4mQ_2 + \frac{4mQ_1^4 Q_2}{ (Q_1^2+ Q_2^2)^2} + \frac{4 Q_1 ^4 Q_2}{(2 Q_1^2+ Q_2^2)^2} - 2Q_1^2 Q_2,
\end{equation}
where $'$ is the derivative with respect to $s$.

From the proof \cite{OY2} of the existence of the Schubart-like periodic orbit $Q_1(s)$, $Q_2(s)$, $P_1(s)$, $P_2(s)$ of periodic $T$, there is a positive constant $R(m)$ such that
\[ Q_1(0)= R(m),\ Q_2(0)=0, \  P_1(0)=0, \ P_2(0)= 2m^{3/2}.\]
These initial conditions correspond to a binary collision of the two inner bodies. By the construction of the periodic orbit, another binary collision of the two inner bodies occurs at $s=T/2$ where the conditions are
\[ Q_1(T/2) = -R(m), \ Q_2(T/2) = 0, \ P_1(T/2)=0, \ P_2(T/2) = -2m^{3/2}.\]
Simultaneous binary collisions correspond to the conditions of the periodic solution when $s=T/4$ and $s=3T/4$, i.e., $Q_1(s)=0$ at these values of $s$. The value of $R(1)$ is approximately $2.29559$. Figure 1 contains a plot of the coordinates $Q_1,Q_2,P_1,P_2$ of the periodic orbit when $m=1$.

\vspace{-0.3in}
\begin{center} 
\myfig{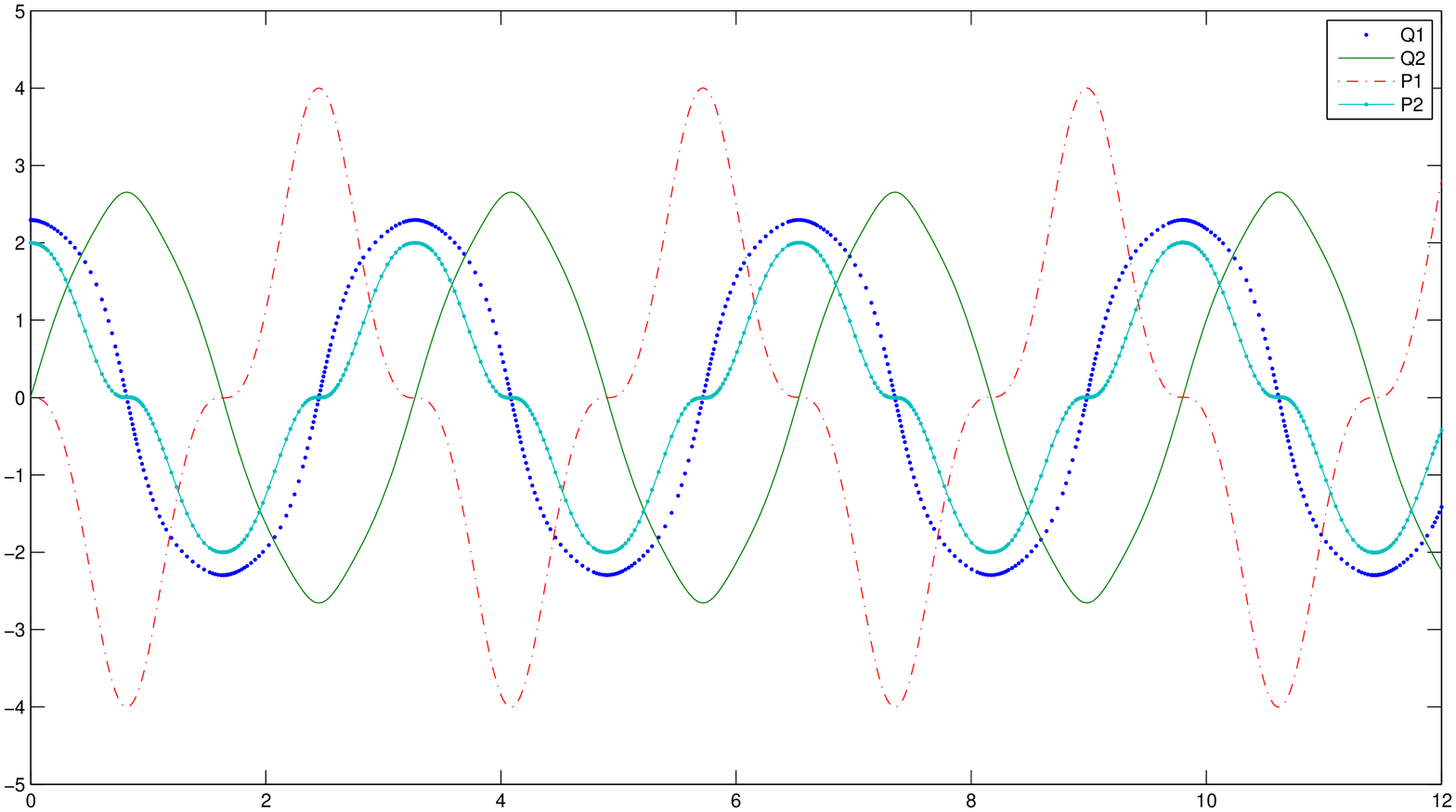}{1.05} \vspace{-0.5in}
\mycaption{The periodic solution in the coordinate system $Q_1, Q_2, P_1, P_2$ when $m=1$.}
\end{center}

\subsection{Stability Reductions using Symmetry}

The Schubart-like periodic solution $\gamma(s)=(Q_1(s),Q_2(s),P_1(s),P_2(s))$ with period $T$ in the collinear problem has two time-reversing symmetries. For
\[ F=\left[ \begin{array}{ccc}
1 & 0  \\
0 & -1
\end{array}\right],\]
the matrix
\[ S= \left[
\begin{array}{ccc}
F & 0  \\
0 & -F
\end{array}
\right]\]
is orthogonal and symmetric: $S^{-1}= S^{T} =S$. It is also an involution, i.e., $S^2=I$. Since $S\gamma(-s+T)$ is a solution of (\ref{e1}) through (\ref{e4}), and since this solution shares the same initial conditions as $\gamma(s)$ at $s=0$ by $T$-periodicity of $\gamma$, uniqueness of solutions implies that the matrix $S$ satisfies
$$\gamma(-s+T)= S  \gamma(s) \text{\ for all } s.$$
Thus $S$ is a time-reversing symmetry of $\gamma(s)$. With $N=1$, conditions (2), (3), and (4) in Lemma \ref{b} are satisfied, and so by Corollary \ref{c}, the monodromy matrix for $\gamma$ satisfies
\begin{equation}\label{3}
X(T)= S X(T/2)^{-1} S^{T} X(T/2)= S X(T/2)^{-1} S X (T/2).
\end{equation}
Consequently, from the above equation and $S^2=I$,
$$\left[S X(T) \right]^2=\left[X(T/2)^{-1} S X(T/2)\right] \left[X(T/2)^{-1} S X(T/2)\right]= I .$$
Since $-S\gamma(-s+T/2)$ is a solution of (\ref{e1}) through (\ref{e4}), and as $-S\gamma(T/2)$ is the same as $\gamma(0)$, uniqueness of solutions implies that the matrix $-S$ satisfies
$$ \gamma(-s+ T/2)= -S \gamma(s) \text{\ for all } s.$$
Thus $-S$ is another time-reversing symmetry of $\gamma(s)$. For $N=2$, conditions (2), (3), and (4) of Lemma \ref{b} are satisfied, and so Corollary \ref{c} implies that
\begin{equation}\label{4}
X(T/2)= S X(T/4)^{-1} S X(T/4).
\end{equation}
For
\[ B=X (T/4),\]
combining equations (\ref{3}) and (\ref{4}) gives
$$X(T)= (SB^{-1} S B)^2$$
With $A=SB^{-1} SB$ and $D=B^{-1} SB$, then
\[ X(T)=A^2=(SD)^2,\]
where $S^2=I$ and $D^2=I$. The two time-reversing symmetries $S$ and $-S$ of $\gamma$ are both involutions, and together they generate a $D_2$ symmetry group for $\gamma$.

\subsection{A Good Basis}
We have reduced the stability analysis to the first quarter of the periodic oribt. Let $Y(s)$ be the fundamental matrix solution to the linearized equations about Schubart-like periodic orbit $\gamma(s)$ with arbitrary initial conditions $Y_0$. Let
\[ B= Y(T/4).\]
By the third remark following Corollary \ref{c}, the matrix $Y_0^{-1} Y(T)$, which is similar to the monodromy matrix $X(T)= Y(T) Y_0^{-1}$, satisfies 
$$Y_0^{-1} Y(T)=\left( \left(Y_0^{-1} S Y_0\right) B^{-1} SB   \right)^{2}.$$
The question of stability reduces to showing that the eigenvalues of
$$W=\left(Y_0^{-1} S Y_0\right) B^{-1} SB$$
are on the unit circle. An appropriate choice of $Y_0$ will simplify the factor $Y_0^{-1}SY_0$ in $W$. Set
\[ \Lambda = \begin{bmatrix} I & 0 \\ 0 & -I\end{bmatrix}.\]

\begin{lemma}\label{Choice1}
There exists $Y_0$ such that
\begin{enumerate}
\item $Y_0$ is orthogonal and symplectic, and
\item $Y_0^{-1} S Y_0= \Lambda$.
\end{enumerate}
\end{lemma}

\begin{proof}
Choose the third column of $Y_0$ to be $\gamma^{\,\prime}(0)/\Vert\gamma^{\,\prime}(0) \Vert = [0\ 1\ 0\ 0]^{T} =e_2$. For $e_3=[0\ 0\ 1\ 0]^T$, the matrix
\[ Y_0= \left[ Je_2, Je_3, e_2, e_3 \right] = \left[
\begin{array}{cccc}
0 & 1 & 0 & 0  \\
0 & 0 & 1 & 0  \\
0 & 0 & 0 & 1  \\
-1 & 0 & 0 & 0  
\end{array}
\right]\]
is orthogonal and symplectic. Since $S={\rm diag}\{1, -1, -1, 1\}$, it follows that $Y_0^{-1} S Y_0$ has the desired form. 
\end{proof}

Setting $D= B^{-1}SB$ and choosing $Y_0$ as constructed in Lemma \ref{Choice1} gives
\[ W=\left(Y_0^{-1} S Y_0\right) B^{-1} SB= \Lambda D.\]
The matrices $\Lambda$ and $D$ are both involutions, i.e., $\Lambda^2=I$, $D^2=I$. From these it follows that
\[ W^{-1}= D \Lambda.\]
Because $B$ is a symplectic matrix, a short computation using the formula for the inverse of a symplectic matrix shows that $D$ has the form
\[ \left[ \begin{array}{ccc} K^{T} & L_1  \\ -L_2 & -K \end{array}\right]\]
for $2\times 2$ matrices $K,L_1,L_2$. It follows that
$$W=\left[ \begin{array}{ccc} I & 0  \\ 0 & -I \end{array} \right] \left[ \begin{array}{ccc} K^{T} & L_1  \\ -L_2 & -K \end{array} \right] = \left[ \begin{array}{ccc} K^{T} & L_1  \\ L_2 & K \end{array} \right],$$
and
$$W^{-1}=  \left[ \begin{array}{ccc} K^{T} & L_1  \\ -L_2 & -K \end{array} \right] \left[ \begin{array}{ccc} I & 0  \\ 0 & -I \end{array} \right]=\left[ \begin{array}{ccc} K^{T} & -L_1  \\ -L_2 & K \end{array}
\right]. $$
Hence,
$$\frac{1}{2} \left( W + W^{-1}\right) =\left[ \begin{array}{ccc} K^{T} & 0  \\ 0 & K \end{array} \right]. $$

We show that the first column of $K$ is $[-1\ 0]^T$. Set $v= Y_0^{-1} \gamma^{\,\prime}(0)$. By the choice of $Y_0$,
\[ v = Y_0^T \gamma^{\,\prime}(0) = \Vert \gamma^{\,\prime}(0)\Vert e_3.\]
Since $S$ is symmetric and $Y_0$ is orthogonal, then by the third remark after Corollary \ref{c},
\[ W = Y_0^{-1}S Y_0 B^{-1}SB = Y_0^{-1} S Y_0 B^{-1}S^T B = Y_0^{T} Y(T/2).\]
Now $\gamma^{\,\prime}(s)$ is a solution of $\dot \xi = JD^2\Gamma(\gamma(s))\xi$ and $\gamma^{\,\prime}(0) = Y(0) Y_0^{-1}\gamma^{\,\prime}(0) = Y(0)v$, and so
$\gamma^{\, \prime}(s) = Y(s) Y_0^{-1} \gamma^{\,\prime}(0) = Y(s)v$. This implies that
\[ Y_0^{-1} \gamma^{\,\prime}(T/2) =  Y_0^T Y(T/2) v = W v.\]
Since $\gamma(s)$ satisfies $\gamma(-s+T/2) = -S\gamma(s)$ for all $s$, then $\gamma^{\,\prime}(-s+T/2) = S\gamma^{\,\prime}(s)$ for all $s$. Setting $s=0$ in this gives $ \gamma^{\, \prime}(T/2)=S\gamma^{\,\prime}(0)$. Since $\gamma^{\,\prime}(0)$ is a nonzero scalar multiple of $e_2$ and since $Se_2=-e_2$, then
\[ Y_0^{-1}\gamma^{\,\prime}(T/2) = Y_0^{-1} S\gamma^{\,\prime}(0) = - Y_0^{-1} \gamma^{\,\prime}(0) = -Y_0^T \gamma^{\,\prime}(0) = -v.\]
Thus $Wv=-v$, implying that $-1$ is an eigenvalue of $W$ and $e_3$ is an eigenvector of $W$ corresponding to this eigenvalue. Thus the first column of $K$ is as claimed. The form of the rest of $K$ comes from the formula for the inverse of a symplectic matrix and the definition of $D$:
\[ K= \left[ \begin{array}{ccc} -1 & *  \\ 0 & c_2^{T} \left( S J c_4 \right) \end{array} \right],\]
where $c_i$ is the ${i}^{\rm th}$ column of $Y(T/4)$.

\subsection{Numerical Calculations}
With an absolute error tolerance of $1\times10^{-12}$, our numerical results for $m=1$ showed that the initial condition 
$$ Q_1(0)= R(1)=2.295592258717, \ \ Q_2(0)=0, \ \  P_1(0)=0, \ \ P_2(0)= 2 $$
leads to a periodic simultaneous binary collision periodic orbit (as in Figure 1) whose period $T$ satisfies $T/4= 0.817348080989685$. Using MATLAB and a Runge-Kutta-Fehlberg algorithm, we computed the columns of the matrix $Y(T/4)$ with an absolute error tolerance of $4\times 10^{-6}$. From this, we got
$$c_2^{T}( S J c_4 ) = 0.598490.$$
For values of $m$ between $0$ and $50$ at $0.01$ increments, we numerically computed the value of $R(m)$ in the initial conditions and the value of the period $T$ (with an absolute error tolerance of $4\times 10^{-6}$), and the values of $c_2^T (SJ c_4)$ (with an absolute error tolerance of $1\times 10^{-6}$). The results of these computations are contained in Figure 2.

A closer look at the numerical data in Figure 2 for where the value of $c_2^T(SJ c_4)$ is close to $1$ gives estimates of the two values of $m$ where the stability of the periodic orbit changes. The first critical value of $m$ is approximately $m=2.83$, and the second critical value of $m$ is approximately $m=35.4$.

The eigenvalues of $K$ are $-1$ and $c_2^T(SJ c_4)$. The eigenvalues of $K$ are distinct for most values of $m$ in $[0,50]$ because of the rigorous numerical estimates we have for $c_2^T(SJ c_4)$. Lemma \ref{06} now implies the following linear stability result.

\begin{theorem}\label{first} There exists small positive constants $\epsilon_i$, $i=1,2,3,4$ such that the periodic simultaneous binary collision orbit in the collinear symmetric four body problem with masses $1$, $m$, $m$, $1$ is linearly stable when $m<2.83-\epsilon_1$ and $35.4+\epsilon_2<m\leq50$, and is linearly unstable when $2.83+\epsilon_3<m<35.4-\epsilon_4$.
\end{theorem}

This result confirms the linear stability analysis of Sweatman \cite{SW2} for $m$ between $0$ and $50$, asserting that the periodic orbit is unstable when $m$ is between $2.83$ and $35.4$. Simulations of the periodic orbit when $m$ is between $2.83$ and $35.4$ indicate that the linear instability is manifested slowly over time.

\begin{center}
\myfig{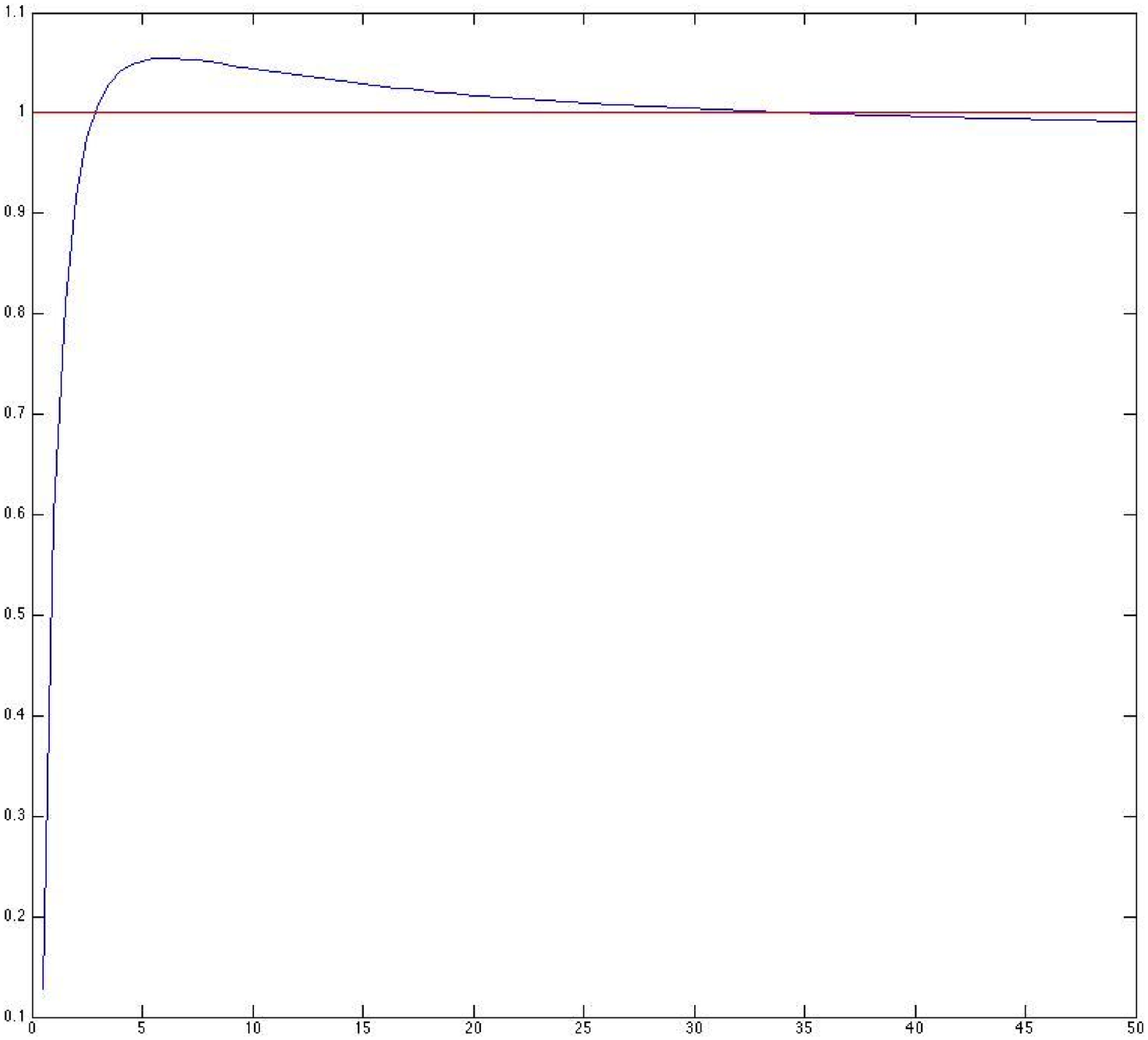}{1.0} \vspace{-0.6in}
\mycaption{The value of $c_2^T(SJ c_4)$ for values of $m$ between $0$ and $50$.}
\end{center}

\section{Linear Stability for the 2D Symmetric Periodic Orbit}

In \cite{OY3}, we proved the existence of a special type of planar periodic solution of $2n$ bodies with equal masses. In this section, we are going to consider the linear stability of this periodic solution when $n=2$. If $(x_1,x_2)$ is the position of the first body, then the positions of the remaining three bodies are $(x_2,x_1)$, $(-x_1,-x_2)$, and $(-x_2,-x_1)$. When each body has mass $m=1$, the Newtonian equations for this planar four-body problem are
\begin{equation*}
(\ddot x_1,\ddot x_2) = -\left[
\frac{(x_1-x_2,x_2-x_1)}{2^{3/2}\vert x_1-x_2\vert^3}
+ \frac{(x_1,x_2)}{4(x_1^2 + x_2^2)^{3/2}}
+ \frac{(x_1+x_2,x_1+x_2)}{2^{3/2}\vert x_1+x_2\vert^3}
\right].
\end{equation*}
The initial conditions for the periodic orbit, and the periodic orbit are illustrated in Figure 3.

\begin{center}
\myfig{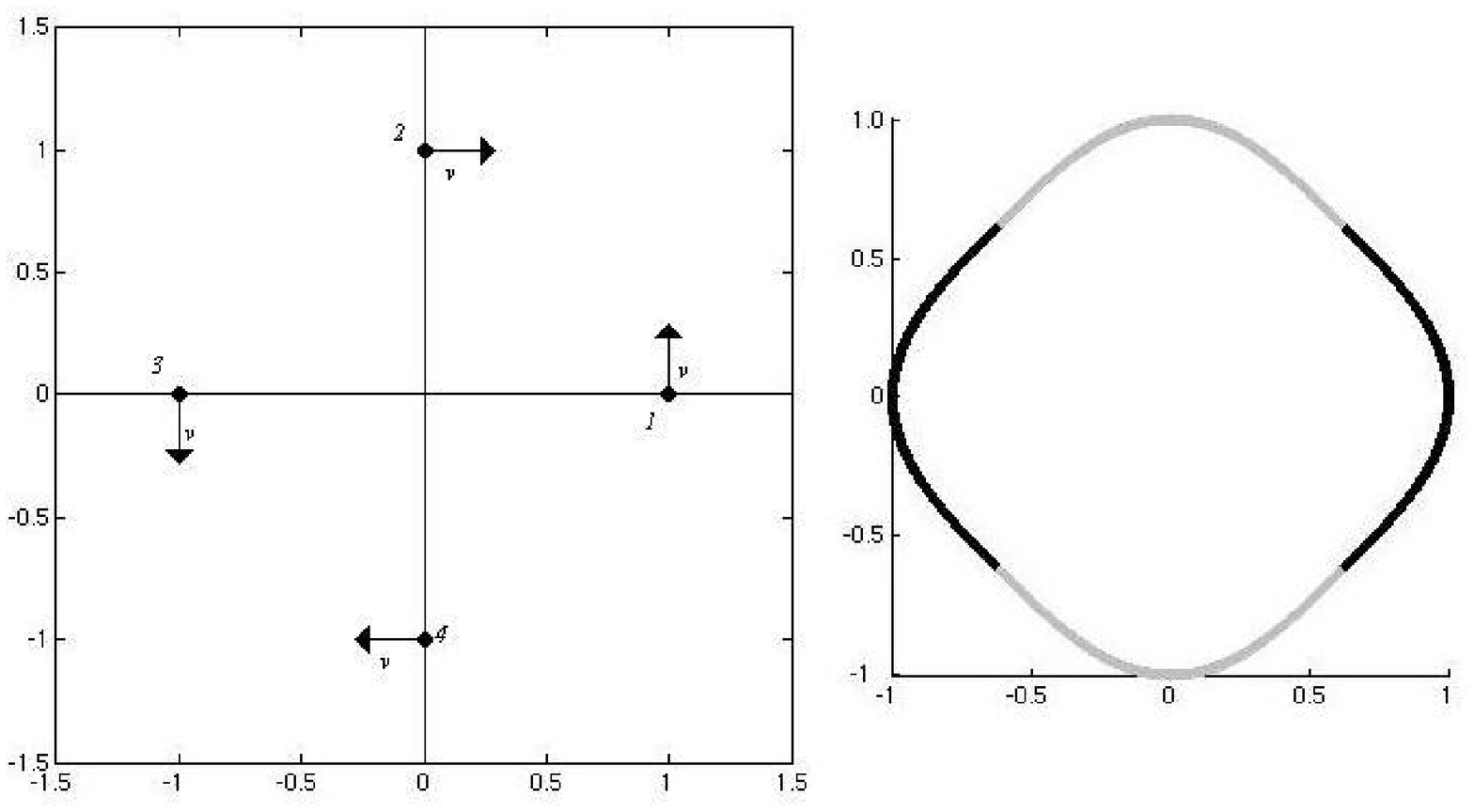}{0.8} \vspace{-0.3in}
\mycaption{On the left are the initial
conditions leading to the four-body two-dimensional periodic SBC
oribt.  On the right is the orbit.}
\end{center}

We adapt Sweatman's approach (\cite{SW}, \cite{SW2}) to regularize this system. The Hamiltonian for this system is
\[ H = \frac{1}{8} \big(w_1^2 + w_2^2\big) - \frac{\sqrt 2}{\vert x_1-x_2\vert} - \frac{\sqrt 2}{\vert x_1+x_2\vert} - \frac{1}{\sqrt{x_1^2+x_2^2}},\]
where $w_1=4\dot x_1$ and $w_2=4\dot x_2$ are the conjugate momentum. In terms of the canonical coordinates $(q_1,q_2,p_1,p_2)$ defined by
\[ q_1 = x_1-x_2,\ \ \ q_2=x_1+x_2,\ \ \ w_1= p_1+p_2,\ \ \  w_2=p_2-p_1,\]
the Hamiltonian becomes
\[ H = \frac{1}{4}\big(p_1^2+p_2^2\big) - \frac{\sqrt 2}{\vert q_1\vert} - \frac{\sqrt 2}{\vert q_2\vert} - \frac{\sqrt 2}{\sqrt{q_1^2+q_2^2}}.\]
The Levi-Civita type of canonical transformation used to regularize the collinear problem now applies to the four body equal mass 2D problem. In terms of the canonical coordinates $(Q_1,Q_2,P_1,P_2)$ defined by
\[ q_i= Q_i^2, \ \ \ P_i = 2Q_ip_i \ \ \ (i=1,2),\]
and the new time variable $s$ defined by
\[ \frac{dt}{ds} = Q_1^2Q_2^2,\]
the Hamiltonian in extended phase space becomes
\begin{equation}\label{00}
\Gamma= \frac{dt}{ds}(H-E)=
\frac{1}{16} (P_1^2Q_2^2 + P_2^2 Q_1^2)- \sqrt{2}(Q_1^2+Q_2^2)
-\frac{\sqrt{2}Q_1^2 Q_2^2 }{\sqrt{Q_1^4+ Q_2^4} } -EQ_1^2 Q_2^2
\end{equation}
where $E$ is the total energy of the Hamiltonian $H$. The differential equations in terms of the new coordinates $\{Q_1, Q_2, P_1, P_2\}$ are
\begin{equation}\label{01}
Q_1' = \frac{1}{8}P_1 Q_2^2
\end{equation}
\begin{equation}\label{02}
Q_2' = \frac{1}{8}P_2 Q_1^2
\end{equation}
\begin{equation}\label{03}
P_1' = -\frac{1}{8}P_2^2 Q_1 + 2\sqrt{2}Q_1+\frac{2\sqrt{2}Q_1 Q_2^2}{\sqrt{Q_1^4+Q_2^4}}-\frac{2\sqrt{2}Q_1^5 Q_2^2}{(Q_1^4+Q_2^4)^\frac{3}{2}} + 2EQ_1Q_2^2
\end{equation}
\begin{equation}\label{04}
P_2' = -\frac{1}{8}P_1^2 Q_2 + 2\sqrt{2}Q_2+\frac{2\sqrt{2}Q_2 Q_1^2}{\sqrt{Q_1^4+Q_2^4}}-\frac{2\sqrt{2}Q_2^5 Q_1^2}{(Q_1^4+Q_2^4)^\frac{3}{2}} + 2EQ_2Q_1^2.
\end{equation}
Unlike the collinear problem, we do not fix the value of $E$ here. As shown in \cite{OY2}, for each $\zeta>0$ there exists $v_0>0$ such that the initial conditions
\begin{equation}\label{05}
Q_1(0)=\zeta, \ \ Q_2(0)=\zeta,  \ \ P_1(0)=-4v_0,  \ \  P_2(0)=4v_0,
\end{equation} 
lead to a periodic solution with a minimal period $T$. From $\Gamma=0$, the value of $E$ is determined by this choice of $\zeta$ and $v_0$. By its construction in \cite{OY3}, this periodic orbit satisfies
\[ Q_1(T/4) = -\zeta, \ \ Q_2(T/4) = \zeta, \ \ P_1(T/4) = -4v_0, \ \ P_2(T/4) = -4v_0.\]
Simultaneous binary collisions correspond to $s=T/8,5T/8$ i.e., when $Q_1(s) = 0$, and to $s=3T/8,7T/8$, i.e., when $Q_2(s)=0$. For $\zeta=1$, $4v_0=2.57486992651942$, and $T/8=1.62047369909693$. Figure 4 illustrates the coordinates $(Q_1,Q_2,P_1,P_2)$ of this periodic solution.

\begin{center}
\myfig{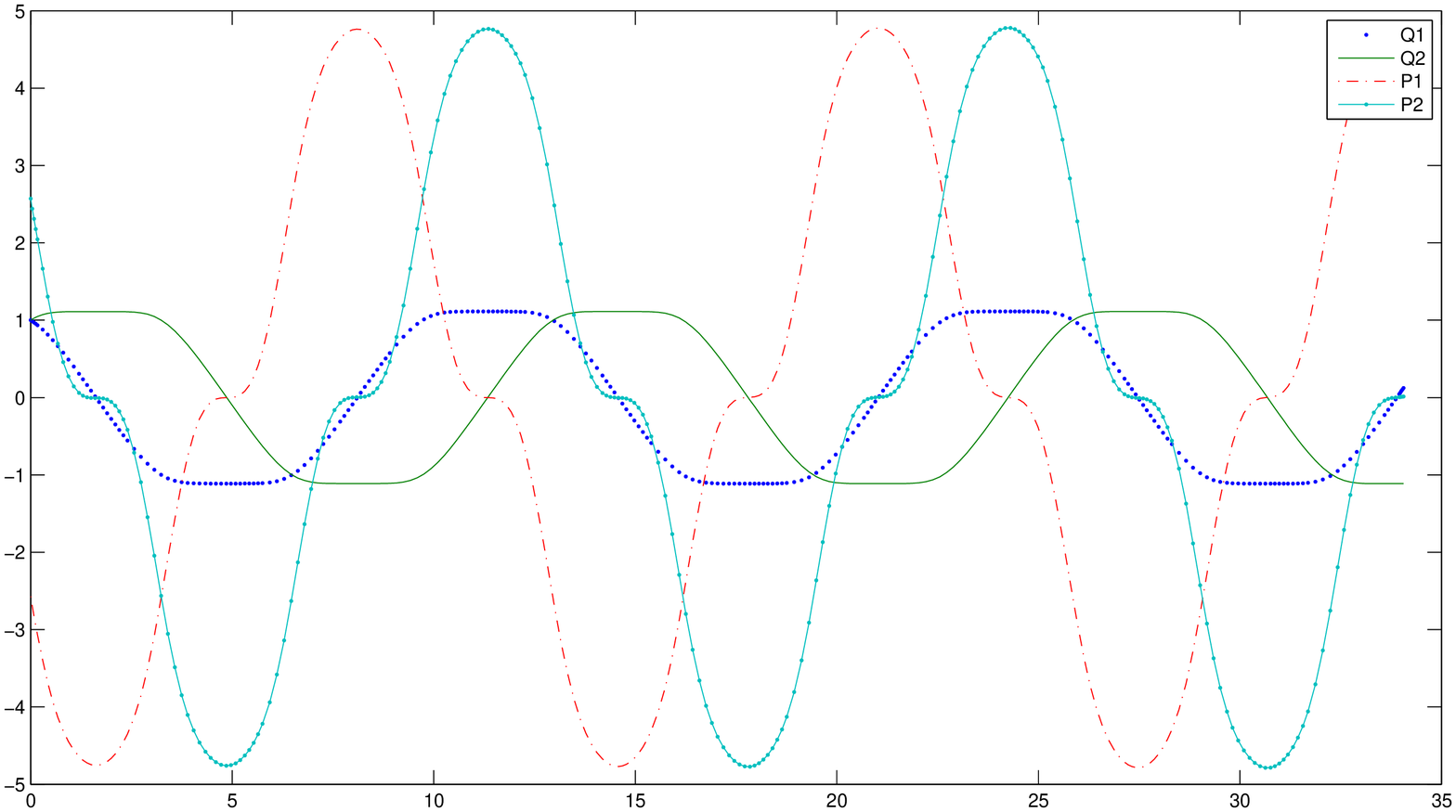}{1} \vspace{-0.5in}
\mycaption{The periodic solution in the coordinate system $Q_1, Q_2, P_1, P_2$.}
\end{center}

\subsection{Stability Reductions using Symmetry}

We will reduce the stability analysis to the first eighth of the periodic orbit. The symmetric periodic 2D orbit
\[ \gamma(t)= (Q_1(t), Q_2(t), P_1(t), P_2(t))\]
with period $T$ has a time-reversing symmetry and a time-preserving symmetry. For
\[ F = \begin{bmatrix} 0 & -1 \\ 1 & 0\end{bmatrix}, \ \ G = \begin{bmatrix} -1 & 0 \\ 0 & 1\end{bmatrix},\]
the matrices
\[ S_F = \begin{bmatrix} F & 0 \\ 0 & F\end{bmatrix}, \ \ S_G = \begin{bmatrix} G & 0 \\ 0 & -G\end{bmatrix}\]
satisfy $S_F^{-1}=S_F^T$, $S_F^2\ne I$, $S_F^3\ne I$, $S_F^4=I$, $S_G^2=I$, $S_G^T= S_G$, and $(S_FS_G)^2=I$. Since $\gamma(s+T/4)$ and $S_F\gamma(s)=(-Q_2(s),Q_1(s),-P_2(s),P_1(s))$ are solutions of (\ref{01}) through (\ref{04}) and share the same initial conditions when $s=0$, uniqueness of solutions implies that
\[ \gamma(s+T/4) = S_F\gamma(s) {\rm\ for\ all\ }s.\]
Thus $S_F$ is a time-preserving symmetry of $\gamma(s)$. With $N=4$, conditions (2), (3), and (4) of Lemma \ref{a} are satisfied, so that Corollary \ref{aa} (with $k=4$) and $S_F^4=I$ imply that
$$X(T)= S_{F}^4\left(S_{F}^{T} X(T/4)\right)^4= \left(S_{F}^{T} X(T/4)\right)^4.$$
Since $\gamma(-s+T/4)$ and $S_G \gamma(s)$ are solutions of (\ref{01}) through (\ref{04}) and share the same initial conditions when $s=0$, uniqueness of solutions implies that
\[ \gamma(-s+T/4) = S_G\gamma(s) {\rm\ for\ all\ } s.\]
Thus $S_G$ is a time-reversing symmetry for $\gamma(s)$. With $N=4$, conditions (2), (3), and (4) of Lemma \ref{b} are satisfied, and so Corollary \ref{c} implies that
$$ X(T/4)= S_G \left[X(T/8)\right]^{-1} S_G^T X(T/8)= S_G \left[X(T/8)\right]^{-1} S_G  X(T/8).$$
Let
\[ B=X(T/8).\]
Combining the factorization of $X(T)$ that involves $S_F$ and the factorization of $X(T/4)$ that involves $S_G$ gives the factorization
$$X(T)= \left(S_F^T S_G B^{-1} S_G B \right)^4.$$
Setting
\[ Q= S_{F}^{T} S_{G}= \left[
\begin{array}{cccc}
0 &  1 & 0 & 0 \\
1 & 0 & 0 & 0 \\
0 & 0 & 0 & -1 \\
0 & 0 & -1 & 0
\end{array}
\right] \]
and $D= B^{-1} S_{G} B$ results in the factorization
\[ X(T)= (QD)^4\]
where $Q$ and $D$ are both involutions. The symmetries $S_F$ and $S_G$ generate a $D_4$ symmetry group for the periodic orbit $\gamma(s)$.

\subsection{A Good Basis}

Let $Y(s)$ be the fundamental matrix solution to the linearized equations about the 2D periodic orbit $\gamma(s)$ with arbitrary initial conditions $Y_0$. Let
\[ B = Y(T/8).\]
By remarks following Corollaries \ref{aa} and \ref{c}, the matrix $Y_0^{-1}Y(T)$, which is similar to the monodromy matrix $X(T)=Y(T) Y_0^{-1}$, satisfies
\[ Y^{-1}_0 Y(T)= (Y^{-1}_0 S^{T}_{F} S_{G} Y_0 B^{-1} S_G B)^4 = (Y_0^{-1} Q Y_0 B^{-1}S_GB)^4.\]
The question of linear stability reduces to showing that the eigenvalues of
\[ W= Y^{-1}_0 Q Y_0 B^{-1} S_G B\]
are on the unit circle. Recall that
\[ \Lambda = \begin{bmatrix} I & 0 \\ 0 & -I\end{bmatrix}.\]

\begin{lemma}\label{choice2} There exists $Y_0$ such that
\begin{enumerate}
\item $Y_0$ is orthogonal and symplectic, and
\item $Y_0^{-1} Q Y_0= \Lambda$.
\end{enumerate}
\end{lemma}

\begin{proof} Choose the third column of $Y_0$ to be 
$$\frac{\gamma^{\,\prime}(0)}{\Vert\gamma^{\,\prime}(0) \Vert }=   \frac{1}{c} \left[
\begin{array}{cccc}
-a & a & b & b 
\end{array}
\right]^{T}   $$
where $a=v_0\zeta^2/2$, $b= E\zeta^3=(2 v_0^2 - 2 \sqrt{2} -1)\zeta$ and $c= \sqrt{2a^2+2b^2}$.  Let $\text{col}_{i}(Y_0)$ denote the $i^{\rm th}$ column of $Y_0$. Define

$$\text{col}_{1} (Y_0)= J \cdot \text{col}_{3} (Y_0)= \frac{1}{c} \left[
\begin{array}{cccc}
b & b & a & -a
\end{array}
\right]^{T} .$$ 

We now choose $\text{col}_{4} (Y_0)$ such that $\text{col}_{4} (Y_0)$ is orthogonal to $\text{col}_{3} (Y_0)$, and $\text{col}_{4} (Y_0)$ is  one of the eigenvectors of Q with respect to its eigenvalue of $-1$. Since the eigenspace of $Q$ corresponding to its eigenvalue of $-1$ is
\[ \text{span} \big\{ \left[ \begin{array}{cccc}
1 & -1 & 0 & 0
\end{array}
\right]^{T} ,\left[\begin{array}{cccc}
0 & 0 & 1 & 1
\end{array}
\right]^{T} \big\},\]
define
 $$\text{col}_{4} (Y_0)= \frac{1}{c} \left[
\begin{array}{cccc}
b & -b & a & a
\end{array}
\right]^{T}$$ 
and
$$\text{col}_{2} (Y_0)= J\cdot \text{col}_{4} (Y_0)= \frac{1}{c} \left[
\begin{array}{cccc}
a & a & -b & b
\end{array}
\right]^{T}. $$
The matrix
$$ Y_0= \frac{1}{c} \left[
\begin{array}{cccc}
b & a & -a & b \\
b & a & a & -b \\
a & -b & b & a \\
-a & b & b & a
\end{array}
\right],$$
 is both symplectic and orthogonal and it satisfies $Y^{-1}_0 Q Y_0=  \Lambda$.
\end{proof}

Setting $D=B^{-1}S_G B$ and choosing $Y_0$ to be the matrix constructed in Lemma \ref{choice2} gives $W= \Lambda D$. The matrices $\Lambda$ and $D$ are involutions (the latter because $S_G^2=I$). As in Section 3.2, $W^{-1} = D\Lambda$, and there is a $2\times 2$ matrix $K$ such that
\[ \frac{1}{2}\big( W+W^{-1}\big) = \begin{bmatrix} K^T & 0 \\ 0 & K\end{bmatrix}.\]

We show that the first column of $K$ is $[1\ 0]^T$. Since $S_G^T= S_G$, $Y_0^{-1}=Y_0^T$, it follows by the third remark following Corollary \ref{c} that
\[ W = Y^{-1}_0 S^{T}_{F} S_{G} Y_0 B^{-1} S_G B=Y^{-1}_0 S^{T}_{F} Y(T/4) = Y_0^{T} S^{T}_{F} Y(T/4).\]
Set $v= Y_0^{-1} \gamma^{\,\prime}(0)$. By the choice of the matrix $Y_0$,
\[v=Y_0^{-1} \gamma^{\,\prime}(0)= Y_0^{T} \gamma^{\,\prime}(0)= \left[
\begin{array}{cccc}
0   \\
0 \\
|| \gamma^{\,\prime}(0) || \\
0
\end{array}
\right] =|| \gamma^{\,\prime}(0) || e_3.\]
Because $\gamma^{\,\prime}(s)$ is a solution to the linearized equation $ \dot{\xi}= JD^2 \Gamma(\gamma(s)) \xi $ and because $\gamma^{\,\prime}(0)= Y(0) Y_0^{-1} \gamma^{\,\prime}(0)$, then $\gamma^{\,\prime}(s)= Y(s) Y_0^{-1} \gamma^{\,\prime}(0)$ for all $s$. Hence,
\begin{equation}\label{11}
Wv=Y^{T}_0 S^{T}_{F} Y(T/4) v=  Y^{T}_0 S^{T}_{F} \gamma^{\,\prime}(T/4).
\end{equation}
Since $\gamma$ satisfies $\gamma(s+T/4)= S_F \gamma(s)$ for all $s$ and $S_F^{-1}=S_F^T$, it then follows that
\[ \gamma^{\,\prime}(s)= S^{-1}_{F} \gamma^{\,\prime}(s+T/4) = S^{T}_{F} \gamma^{\,\prime}(s+T/4).\]
Setting $s=0$ in this gives $\gamma^{\,\prime}(0)=S^{T}_{F} \gamma^{\,\prime}(T/4)$, and consequently that
\begin{equation}\label{12}
Y^{T}_0 S^{T}_{F} \gamma^{\,\prime}(T/4)= Y^{T}_0 \gamma^{\,\prime}(0)= Y^{-1}_0 \gamma^{\,\prime}(0)= v.
\end{equation}
Equations (\ref{11}) and (\ref{12}) now combine to show that $Wv =v$, i.e, that $1$ is an eigenvalue of $W$ and $e_3$ is an eigenvector for $W$ corresponding to this eigenvalue. The  first column of $K$ is as claimed. The rest of $K$ comes from the formula for the inverse of a symplectic matrix and the definition of $D$:
$$K=\left[
\begin{array}{ccc}
1 & * \\
0 &  c_2^{T} (S_G J c_4) 
\end{array}
\right] ,$$ 
where $c_i$ is the $i^{\rm th}$ column of $B= Y(T/8)$.

\subsection{Numerical Calculations}

Having not fixed $E$, we used an invariant scaling of the coordinates and time in equations (\ref{01}) through (\ref{04}) to preselect a period $T$ before numerically computing the initial conditions for a periodic simultaneous binary collision orbit. For $\epsilon>0$, if $Q_1(s)$, $Q_2(s)$, $P_1(s)$, $P_2(s)$ is a periodic simultaneous collision orbit of equations (\ref{01}) through (\ref{04}), then replacing $E$ with $\epsilon^{-2} E$ shows that $\epsilon Q_1(\epsilon s)$, $\epsilon Q_2(\epsilon s)$, $P_1(\epsilon s)$, $P_2(\epsilon s)$ is also a periodic simultaneous binary collision orbit with energy $\epsilon^{-2}E$ and period  $\epsilon^{-1}T$. Furthermore, it is straight-forward to show that monodromy matrices for the periodic simultaneous binary collision orbits corresponding to values of $\epsilon\ne  1$ are all similar to that for $\epsilon=1$. Thus the linear stability of a periodic simultaneous binary collision orbit for one $\epsilon>0$ implies the linear stability of the periodic simultaneous binary collision orbits for all $\epsilon>0$.

We rigorously computed the value of $c_2^T(S_G J c_4)$ for the periodic simultaneous binary collision orbit whose period is $T=8$. This means that the first time of a simultaneous binary collision for this orbit is at  $s=1$. We set $Q_1(0)=Q_2(0)=\xi$ and $-P_1(0)=P_2(0)=\eta$, and defined a function $F(\xi,\eta)$ to be equal to the vector quantity $(Q_1(1),P_2(1))$. We used Newton's method and a good initial guess to find a root $(\xi,\eta)$ of $F$. This involved computing the Jacobian of $F$ which was done using the linearized equations. With an absolute error tolerance of $6\times 10^{-11}$, this  numerical method shows that the initial conditions
$$ Q_1(0)= Q_2(0)=1.62047369909693, \ \  -P_1(0)= P_2(0)= 2.57486992651942, $$
lead to a periodic solution with a period of $T=8$, and a value of $E\approx-1.142329388$. Using MATLAB and a Runge-Kutta-Fehlberg algorithm, we computed the columns of the matrix $Y(T/8)$ with an absolute error tolerance of $2.5\times 10^{-12}$. From this we got
$$ c_2^{T} \left( S_{G} J c_4 \right)= -0.68024151010592. $$
Using the scaling of coordinates and time described above, the initial conditions for the periodic simultaneous binary collision orbit shown in Figures 3 and 4 are
$$ Q_1(0)= Q_2(0)=1, \ \  -P_1(0)= P_2(0)=2.57486992651942 $$
with a period $T$ satisfying $T/8=1.62047369909693$, and energy $E\approx-2.999682732$.

For the periodic simultaneous binary collision orbit, the rigorous estimate of the eigenvalue $c_2^T (S_G J c_4)$ of $K$ and its distinctiveness from the eigenvalue $1$ of $K$ combine with Lemma \ref{06} to give the following stability result.

\begin{theorem} The periodic simultaneous binary collision orbit in the 2D-symmetric equal mass four-body problem is linearly stable.\end{theorem}

When $c_2^T(S_G Jc_4)$ is real and between $-1$ and $1$, it is the real part of an eigenvalue with unit modulus for $W$ (see \cite{RO}). For the periodic simultaneous collision orbit, the real part of $\exp(3\pi i/ 4)$, that is $-(1/2)\sqrt 2$, is fairly close to the rigorously estimated value of $c_2^{T} \left( S_{G} J c_4 \right)$. Raising $\exp(3\pi i/ 4)$ to the fourth power gives $\exp(3\pi i)=-1$, and so two of the eigenvalues of the monodromy matrix of the periodic simultaneous binary collision orbit are close to $-1$. The symmetry reductions used to compute the eigenvalues over just one-eighth of the period and the rigorous estimate of $c_2^{T} \left( S_{G} J c_4 \right)$ showing that it is clearly between $-1$ and $1$, assures the linear stability of the periodic simultaneous binary collision orbit.

\vspace{0.1in}

\noindent{\bf Acknowledgments}. Gareth Roberts thanks the Department of Mathematics at Brigham Young University for hosting him. Lennard Bakker, Tiancheng Ouyang, Duokui Yan, and Skyler Simmons thank Gareth Roberts for his visit and collaboration.

\end{document}